\documentclass[11pt]{amsart}

\newif\ifdraft\draftfalse

\usepackage{amssymb}
\usepackage{amsthm}
\usepackage{eucal}
\usepackage[english]{babel}
\usepackage{nicefrac}
\usepackage{times}
\usepackage{latexsym,amscd,amsmath,amsthm,amssymb}

\makeatletter
\def\@begintheorem#1#2[#3]{%
    \def\naam{#1}
  \deferred@thm@head{\the\thm@headfont \thm@indent
    \@ifempty{#1}{\let\thmname\@gobble}{\let\thmname\@iden}%
    \@ifempty{#2}{\let\thmnumber\@gobble}{\let\thmnumber\@iden}%
    \@ifempty{#3}{\let\thmnote\@gobble}{\let\thmnote\@iden}%
    \thm@swap\swappedhead\thmhead{#1}{#2}{#3}%
    \the\thm@headpunct
    \thmheadnl 
    \hskip\thm@headsep
  }%
  \ignorespaces}
\makeatother

\newcommand{\kantlijndraft}[1]{\ifdraft\hspace{-\lastskip}%
\vadjust{\vspace{-1mm}\smash{\llap{{\tt #1}\hspace{8mm}}}\vspace{1mm}}\fi}

\def\voegToe#1#2#3{\immediate\write1{\string\newlabel{#1}{{#2}{#3}}}}

\newcommand{\thlabel}[1]{\voegToe{#1}{\naam\noexpand~\thetheorem}{\thepage}\kantlijndraft{#1}}

\makeatletter
\renewcommand{\label}[1]{\voegToe{#1}{\@currentlabel}{\thepage}\kantlijndraft{#1}}
\makeatother

\newtheorem{theorem}{Theorem}[section]
\newtheorem{lemma}[theorem]{Lemma}
\newtheorem{corollary}[theorem]{Corollary}
\newtheorem{question}[theorem]{Question}

\theoremstyle{definition}

\newtheorem{definition}[theorem]{Definition}
\theoremstyle{remark}

\numberwithin{equation}{section}

\newtheorem{claim2}{\sc Claim}



\newcommand{\sse}{\subseteq}						
\newcommand{\minus}{\backslash}						
\newcommand{\Un}{\bigcup}							
\newcommand{\un}{\cup}								
\newcommand{\Meet}{\bigcap}							
\newcommand{\meet}{\cap}							
\newcommand{\es}{\varnothing}						

\newcommand{\cl}[1]{\ensuremath{\overline{#1}}}

\newcommand{\scr}[1]{\ensuremath{\mathcal{#1}}}

\def\cprime{$'$}
\def\cont{\mathfrak{c}}

\def\arhangelskii{Arhangel{\cprime}ski{\u\i}}

\begin{document}

\title{More on cardinality bounds involving the weak Lindel\"of degree}

\author{A. Bella}\address{Department of Mathematics, University of Catania, Viale A. Doria 6, 95125 Catania, Italy}
\email{bella@dmi.unict.it}
\author{N. Carlson}\address{Department of Mathematics, California Lutheran University, 60 W. Olsen Rd, MC 3750, 
Thousand Oaks, CA 91360 USA}
\email{ncarlson@callutheran.edu}
\author{I. Gotchev}\address{Department of Mathematical Sciences, Central Connecticut State University, 1615 Stanley Street, 
New Britain, CT 06050 USA}
\email{gotchevi@ccsu.edu}



\begin{abstract} 
We give several new bounds for the cardinality of a Hausdorff topological space $X$ involving the weak Lindel\"of degree $wL(X)$. In particular, we show that if $X$ is extremally disconnected, then $|X|\leq 2^{wL(X)\pi\chi(X)\psi(X)}$, 
and if $X$ is additionally power homogeneous, then $|X|\leq 2^{wL(X)\pi\chi(X)}$. We also prove that if $X$ is an almost Lindel\"of space with a strong $G_\delta$-diagonal 
of rank 2, then $|X|\leq 2^{\aleph_0}$;  that if $X$ is a star-cdc space with a $G_\delta$-diagonal of rank 3, then $|X| \le 2^{\aleph_0}$; and if $X$ is any normal star-cdc space $X$ with a $G_\delta$-diagonal of rank 2, then $|X|\leq 2^{\aleph_0}$. 
Several improvements of results in~\cite{BC2018} are also given. We show that if $X$ is locally compact, then $|X|\leq wL(X)^{\psi(X)}$ and that $|X|\leq wL(X)^{t(X)}$ if $X$ is additionally power homogeneous. We also prove that 
$|X|\leq 2^{\psi_c(X)t(X)wL(X)}$ for any space with a $\pi$-base whose elements have compact closures and that the stronger inequality $|X|\leq wL(X)^{\psi_c(X)t(X)}$ is true when $X$ is locally $H$-closed or locally Lindel\"of.
\end{abstract}

\maketitle

\section{Introduction}\label{S1}
In this study we establish several new cardinality bounds for Hausdorff topological spaces involving the weak Lindel\"of degree and related cardinal invariants. Recall that the \emph{weak Lindel\"of degree} $wL(X)$ of a space $X$ is the least infinite cardinal 
$\kappa$ such that every open cover of $X$ has a $\kappa$-subfamily with dense union. When $wL(X)=\aleph_0$, the space $X$ is called \emph{weakly Lindel\"of}. Note that $wL(X)\leq c(X)$ for any space $X$, where $c(X)$ is the cellularity of $X$. 

By finding bounds for the cellularity $c(X)$ of an extremally disconnected space $X$, we show in Corollary~\ref{edcard} that $|X|\leq 2^{wL(X)\pi\chi(X)\psi(X)}$, for any extremally disconnected space $X$, where $\psi(X)$ is the 
pseudocharacter of $X$ and $\pi\chi(X)$ is the $\pi$-character of $X$. (A Hausdorff topological space $X$ is \emph{extremally disconnected} if the closure of every open subset of $X$ is open.) In Corollary~\ref{edhomog} the related bound 
$2^{wL(X)\pi\chi(X)}$ is given for the cardinality of any power homogeneous extremally disconnected space. Recall that a space $X$ is \emph{homogeneous} if for every $x,y\in X$ there exists a homeomorphism $h:X\to Y$ such that $h(x)=y$. 
$X$ is \emph{power homogeneous} if there exists a cardinal $\kappa$ such that $X^\kappa$ is homogeneous. 

In Section \ref{S3} we explore cardinal inequalities for spaces with $G_\delta$-diagonals of rank 2, and related notions (see Section \ref{S3} for definitions). We recall that Buzyakova in \cite{Buz06} showed that any ccc space with a regular $G_\delta$-diagonal has cardinality not 
exceeding $2^{\aleph_0}$. (Note that every space $X$ with a regular $G_\delta$-diagonal is a Urysohn space i.e. a space such that every two distinct points can be separated by disjoint closed neighborhoods.)
Buzyakova's inequality was extended to higher cardinalities by Gotchev in~\cite{Got} as follows: If $X$ is a Urysohn space, then $|X|\leq 2^{c(X)\overline{\Delta}(X)}$, where $\overline{\Delta}(X)$ is the regular diagonal degree of $X$. We recall that 
the \emph{regular diagonal degree} $\overline{\Delta}(X)$ of a Urysohn space $X$ is the minimal infinite cardinal $\kappa$ such that $X$ has a regular $G_\kappa$-diagonal. By replacing $c(X)$ with $wL(X)$, in the same paper, Gotchev proved that 
if $X$ is a Urysohn space, then $|X|\leq 2^{\overline{\Delta}(X)\cdot 2^{wL(X)}}$. Therefore, every weakly Lindel\"of space with a regular $G_\delta$-diagonal has cardinality at most $2^\cont$.
Bella in \cite{Bel87} showed that a ccc space with a $G_\delta$-diagonal of rank 2 has cardinality at most $2^{\aleph_0}$. In relation to these results in Theorem~\ref{TG} we show that if $X$ is an almost 
Lindel\"of space with a strong $G_\delta$-diagonal of rank 2, then $|X|\leq 2^{\aleph_0}$. We recall that the \emph{almost Lindel\"of degree} of a space $X$, denoted $aL(X)$, is the smallest infinite cardinal $\kappa$ such that for every open cover 
$\scr V$ of $X$, there is a subcollection $\scr{V}_0$ of $\scr V$ such that $|\scr{V}_0|\le\kappa$ and $\{\overline{V}:V\in \scr{V}_0\}$ covers $X$. When $aL(X)=\aleph_0$, the space $X$ is called \emph{almost Lindel\"of}.

Recall that a space $X$ has \emph{countable discrete cellularity} (cdc) if every discrete family of open sets in $X$ is countable. We show in Theorem~\ref{TG2} 
that if $X$ is a star-cdc space with a $G_\delta$-diagonal of rank 3, then $|X| \le 2^{\aleph_0}$, and in Theorem~\ref{nscdc} that any normal star-cdc space $X$ with a $G_\delta$-diagonal of rank 2 satisfies $|X|\leq 2^{\aleph_0}$. 

In 1978, Bell, Ginsburg, and Woods showed in \cite{BGW1978} that $|X|\leq 2^{wL(X)\chi(X)}$ for any normal $T_1$-space $X$, where $\chi(X)$ is the character of the space $X$. 
It is still an open question of Bell, Ginsburg and Woods if the inequality $|X|\leq 2^{wL(X)\chi(X)}$ holds true for regular, $T_1$-spaces \cite{BGW1978} and even for Tychonoff spaces \cite{BC2018}.
Therefore, it is interesting to know if there are other classes of spaces for which Bell, Ginsburg and Woods' inequality is true. Here we mention several such classes of spaces closely related to our results in Section \ref{S4}. For results related to 
Bell, Ginsburg, and Woods' question using topological games see \cite{ABS}.

Gotchev proved in \cite{Got} that for any Urysohn space $X$, $|X|\leq wL(X)^{\overline{\Delta}(X)\chi(X)}$. Therefore, if $X$ is a space with a regular $G_\delta$-diagonal, then $|X| \le wL(X)^{\chi(X)}$. Notice that this inequality improves significantly the 
above-mentioned Bell, Ginsburg and Woods' inequality for the class of normal spaces with regular $G_\delta$-diagonals.
(To see that take a discrete space $Y$ with cardinality of the continuum $\cont$. Then $\chi(Y)=\omega$, $wL(Y)=\cont$ and according to Bell, Ginsburg and Woods' inequality $|Y|\le 2^\cont$, while Gotchev's inequality gives $|Y|\le\cont^\omega=\cont$.)

Bella and Carlson in \cite{BC2018} found more classes of spaces $X$ for which Bell, Ginsburg and Woods' inequality holds true, including locally compact, locally Lindel\"of, locally ccc, regular locally normal $T_1$-spaces (\cite[Theorem 2.12]{BC2018}), 
and also quasiregular or Urysohn spaces with a dense set of points with neighborhoods that are $H$-closed, normal, Lindel\"of, or ccc (\cite[Corollary 2.7]{BC2018}). 

In \cite{BGW1978}, Bell, Ginsburg, and Woods constructed an example of a zero-dimentional Hausdorff (and therefore Tychonoff) space $Y$ such that $|Y|> 2^{wL(Y)t(Y)\psi(Y)}$ and asked if the inequality $|X|\le 2^{wL(X)t(X)\psi(X)}$ is true for every 
normal $T_1$-space $X$. While this question, in its full generality, is still open, partial answers are known. For example, in~\cite{BC2018}, it was shown that $|X|\leq 2^{wL(X)t(X)\psi(X)}$ for any 
regular $T_1$-space with a $\pi$-base whose elements have compact closures. We give an improvement of this result in Theorem~\ref{TBCG} where we show that $|X|\leq 2^{wL(X)t(X)\psi_c(X)}$ for any 
Hausdorff space with a $\pi$-base whose elements have compact closures. (Note that in regular spaces the pseudocharacter $\psi(X)$ is equal to the closed pseudocharacter $\psi_c(X)$).

In Section \ref{S4} we also show that even the stronger inequality $|X|\leq wL(X)^{\psi_c(X)t(X)}$ is true when the Hausdorff space $X$ is locally $H$-closed (Theorem \ref{th4.2}) or locally Lindel\"of (Theorem \ref{th4.4}).
As corollaries of these results we obtain refinements of some of the results in~\cite{BC2018} mentioned above similar to Gotchev's inequality for spaces with regular $G_\delta$-diagonals. 
For example, while in the paper \cite{BC2018} it was proved that $|X|\leq 2^{wL(X)\psi(X)}$, for any locally compact Hausdorff space, we show in Corollary~\ref{lc} that in fact $|X|\leq wL(X)^{\psi(X)}$ for such spaces. 
Similarly, we prove in Theorem~\ref{lcph} that if $X$ is a locally compact power homogeneous Hausdorff space, then $|X|\leq wL(X)^{t(X)}$. This improves the inequality $|X|\leq 2^{wL(X)t(X)}$ proved in \cite{BC2018} under the same hypotheses.

In this work we assume all spaces are Hausdorff. See~\cite{EN} or~\cite{Juhasz} for any undefined notions.

\section{A cardinal inequality for extremally disconnected spaces}\label{S2}

We establish bounds for the cardinality of any extremally disconnected space in Corollay~\ref{edcard2} and Corollay~\ref{edcard}. This is a consequence of the following cardinal inequalities for the cellularity $c(X)$ of an extremally disconnected space. 
We recall that a family $\scr V$ of non-empty open sets in a space $X$ is a local $\pi$-base at a point $x \in X$ if for every open neighborhood $U$ of $x$ there is $V \in \scr V$ such that $V \subset U$. The minimal infinite cardinal $\kappa$ such that 
for each $x \in X$ there is a collection $\{V(\eta,x):\eta<\kappa\}$ of non-empty open subsets of $X$ which is a local $\pi$-base for $x$ is called \emph{$\pi$-character} of $X$ and is denoted by $\pi\chi(X)$.

\begin{theorem}\label{ed}
Let $X$ be an extremally disconnected space. Then
\begin{itemize}
\item[(a)] $c(X)\leq wL(X)t(X)$ and
\item[(b)] $c(X)\leq wL(X)\pi\chi(X)$.
\end{itemize}
\end{theorem}

\begin{proof}
For (a), let $\kappa=wL(X)t(X)$ and for (b) let $\kappa= wL(X)\pi\chi(X)$. In either case suppose $c(X)\geq\kappa^+$. 
Then there exists a cellular family of cardinality at least $\kappa^+$ which is contained in a maximal cellular family 
$\scr{C}$. We have $|\scr{C}|\geq\kappa^+$ and $X=\cl{\Un\scr{C}}$. 

Let $D=\cl{\Un\scr{C}}\minus\Un\scr{C}$. If $D=\emptyset$, then for every $x\in X$ there is $C_x\in\scr{C}$ such that 
$x\in C_x$. Then for each $x\in X$ we let $\scr{C}_x=\{C_x\}$. Now let $D\ne\emptyset$ and let $x\in D$. 

For (a), there exists $A_x\sse\Un\scr{C}$ such that $|A_x|\leq\kappa$ and $x\in \cl{A_x}$. For every $a\in A_x$ there 
exists $C(x,a)\in\scr{C}$ such that $a\in C(x,a)$. Let $\scr{C}_x=\{C(x,a):a\in A_x\}$. Note that 
$|\scr{C}_x|\leq |A_x|\leq\kappa$ and $x\in \cl{A_x}\sse\cl{\Un\scr{C}_x}$. 

For (b), let $\scr{P}$ be a local $\pi$-base at $x$ such that $|\scr{P}|\leq\kappa$. Since $X=\cl{\Un\scr{C}}$, for every 
$U\in\scr{P}$ there exists $C_U\in\scr{C}$ such that $C_U\meet U\neq\es$. Let $\scr{C}_x=\{C_U:U\in\scr{P}\}$ and note 
that $|\scr{C}|\leq\kappa$. It is straightfoward to see that $x\in\cl{\Un\scr{C}_x}$. 

In either case (a) or (b), we see from the above that for every $x\in D$ there exists $\scr{C}_x\in[\scr{C}]^{\leq\kappa}$ 
such that $x\in\cl{\Un\scr{C}_x}$. As each $\cl{\Un\scr{C}_x}$ is open in the extremally disconnected space $X$, we have 
that $\scr{C}\un\left\{\cl{\Un\scr{C}_x}: x\in D\right\}$ is an open cover of $X$. Therefore there exist 
$\scr{B}\sse [\scr{C}]^{\leq\kappa}$ and $B\in[D]^{\leq\kappa}$ such that 
$\Un\scr{B}\un\Un\{\cl{\Un\scr{C}_x}:x\in B\}$ is dense in $X$. 

Let $\scr{D}=\scr{B}\un\{\scr{C}_x:x\in B\}$. Note that $|\scr{D}|\leq\kappa$. Let $C\in\scr{C}\minus\scr{D}$. Then 
$C\meet\cl{\Un\scr{B}}\neq\es$ or $C\meet \cl{\Un_{x\in B}\cl{\Un\scr{C}_x}}\neq\es$. If 
$C\meet\cl{\Un\scr{B}}\neq\es$, then $C\meet (\Un\scr{B})\neq\es$. This is a contradiction as $\scr{C}$ is a pairwise 
disjoint family. If $C\meet \cl{\Un_{x\in B}\cl{\Un\scr{C}_x}}\neq\es$, then 
$C\meet\left(\Un_{x\in B}\cl{\Un\scr{C}_x}\right)\neq\es$. Then there exists $x\in B$ such that 
$C\meet\cl{\Un\scr{C}_x}\neq\es$, and thus $C\meet(\Un\scr{C}_x)\neq\es$. This is also a contradiction. Therefore 
$c(X)\leq\kappa$.
\end{proof}

\begin{corollary}\label{edcard2}
If $X$ is extremally disconnected, then $|X|\leq\pi\chi(X)^{wL(X)t(X)\psi_c(X)}$.
\end{corollary}

\begin{proof}
It was shown by Sun in \cite{SH88} that if $X$ is a Hausdorff space, then 
$|X|\le\pi\chi(X)^{c(X)\psi_c(X)}$. Using Theorem~\ref{ed}(a) we can replace in that inequality $c(X)$ with $wL(X)t(X)$ to obtain the required inequality.
\end{proof}

\begin{corollary}\label{edcard}
If $X$ is extremally disconnected, then $|X|\leq 2^{wL(X)\pi\chi(X)\psi(X)}$.
\end{corollary}

\begin{proof}
As $X$ is extremally disconnected and Hausdorff, we have that the semiregularization $X_s$ of $X$ is extremally 
disconnected and $T_3$. (See, for example, 6P(3) in \cite{por88}). It was shown by \v{S}apirovski{\u\i} in \cite{Sap74} 
that $|Y|\leq\pi\chi(Y)^{c(Y)\psi(Y)}$ for any $T_3$-space $Y$. In addition, it was shown in \cite{car07} that 
$c(Y)=c(Y_s)$ and $\pi\chi(Y_s)\leq\pi\chi(Y)$ for any Hausdorff space $Y$. Using Theorem~\ref{ed}(b), the 
aforementioned facts, and that $\psi(Y_s)\leq\psi(Y)$ for any Hausdorff space, we have
\begin{align}
|X|=|X_s|&\leq \pi\chi(X_s)^{c(X_s)\psi(X_s)}\notag\\
&\leq\pi\chi(X)^{c(X)\psi(X)}\notag\\
&\leq\pi\chi(X)^{wL(X)\pi\chi(X)\psi(X)}\notag\\
&=2^{wL(X)\pi\chi(X)\psi(X)}.\notag
\end{align}
\end{proof}

It was shown in \cite{CR08} that $|X|\leq 2^{c(X)\pi\chi(X)}$ for every power homogeneous Hausdorff space $X$. Using Theorem~\ref{ed}(b) above we obtain the following corollary.

\begin{corollary}\label{edhomog}
If $X$ is an extremally disconnected, power homogeneous space, then $|X|\leq 2^{wL(X)\pi\chi(X)}$.
\end{corollary}

Recall that a space $X$ is \emph{H-closed} if every open cover of $X$ has a finite subfamily with dense union. We observe 
that it follows from Corollary~\ref{edhomog} that an $H$-closed, extremally disconnected, power homogeneous space has 
cardinality at most $2^{\pi\chi(X)}$. However, it was shown in \cite[Theorem 3.3]{car07} that an infinite $H$-closed 
extremally disconnected space cannot be power homogeneous. This latter result is an extension of a result of 
Kunen~\cite{K90} that an infinite compact $F$-space is not power homogeneous.

A Tychonoff space is \emph{basically disconnected} if the closure of each of its cozero sets is open. It is clear that every extremally disconnected Tychonoff space is basically disconnected, and that every basically disconnected space is an $F$-space.

\begin{theorem}
A weakly Lindel\"of basically disconnected Tychonoff space of countable tightness has countable cellularity.
\end{theorem}

\begin{proof}
Suppose that there exists a family $\{U_\alpha:\alpha<\omega_1\}$ of pairwise disjoint non-empty open sets. Of course, we 
may assume that each $U_\alpha$ is a cozero set. As the collection of cozero sets is closed under countable unions and $X$ 
is basically disconnected, for every $\alpha<\omega_1$ the set $V_\alpha=\cl{\Un\{U_\beta:\beta<\alpha\}}$ is clopen. 
Moreover, the countable tightness of $X$ guarantees that even the set $V=\Un\{V_\alpha:\alpha<\omega_1\}$ is clopen. 
Thus $V$ is weakly Lindel\"of as that property is hereditary on regular-closed sets. However, it is straightforward to see that 
no countable subfamily of the cover $\{V_\alpha:\alpha<\omega_1\}$ has dense union in $V$, which is a contradiction.
\end{proof}

It is stated in Problem 6L(8) in~\cite{por88} that if a Tychonoff space has countable cellularity, then it is an $F$-space if and 
only if it is extremally disconnected. As basically disconnected spaces are $F$-spaces, we have the following corollary.

\begin{corollary}
A weakly Lindel\"of Tychonoff space of countable tightness is basically disconnected if and only if it is extremally disconnected.
\end{corollary}

\section{Cardinal inequalities for spaces with a (strong) $G_\delta$-diagonal of rank 2}\label{S3}
Given a family of sets $\scr{B}$ and a set $A$ in a space $X$, the star of $\scr{B}$ around $A$ is the set 
$St(A,\scr{B})=\Un\{B\in\scr{B}:A\meet B\neq\es\}$. If $A=\{a\}$ we write $St(a,\scr{B})$.

A space $X$ has a \emph{$G_\delta$-diagonal} (\emph{strong $G_\delta$-diagonal}) provided that there exists a sequence 
$\{\scr{U}_n:n<\omega\}$ of open covers of $X$ such that for every $x\in X$ we have 
$\Meet\{St(x,\scr{U}_n):n<\omega\}=\{x\}$ 
$\left(\Meet\{\overline{St(x,\scr{U}_n)}:n<\omega\} = \{x\}\right)$.

A space $X$ has a \emph{$G_\delta$-diagonal of rank 2} (\emph{strong $G_\delta$-diagonal of rank 2}) provided that 
there exists a sequence $\{\scr{U}_n:n<\omega\}$ of open covers of $X$ such that  
$\Meet\{St(St(x,\scr{U}_n),\scr{U}_n):n<\omega\}=\{x\}$  
$\left(\Meet\{\overline{St(St(x,\scr{U}_n),\scr{U}_n)}:n<\omega\}=\{x\}\right)$, whenever $x\in X$.
Similarly, for every natural number $n$, spaces with \emph{$G_\delta$-diagonal of rank $n$} 
(\emph{strong $G_\delta$-diagonal of rank $n$}) are defined .

$X$ has a regular $G_\delta$-diagonal if the diagonal of $X$ is the intersection of countably many closed neighborhoods in 
$X\times X$. Equivalently, $X$ has a regular $G_\delta$-diagonal if there exists a sequence $\{\scr{U}_n:n<\omega\}$ of 
open covers of $X$ such that if $x, y\in X$, $x\neq y$, then there is  $n<\omega$ and sets $V_x,V_y\in\scr{U}_n$ 
containing $x$ and $y$, respectively, such that $V_y\cap\overline{St(V_x,\scr{U}_n)}=\emptyset$ \cite{Zenor72}.

If a space has a strong $G_\delta$-diagonal of rank 2, then it has a $G_\delta$-diagonal of rank 2 and a regular $G_\delta$-diagonal. On the other hand, both regular and rank 2 $G_\delta$-diagonal imply strong $G_\delta$-diagonal.

The two intermediate notions seems to play a parallel role. For instance, two remarkable facts are:

\begin{theorem}[{Buzyakova,~\cite{Buz06}}]\label{buz}
A ccc space with a regular $G_\delta$-diagonal has cardinality not exceeding $2^{\aleph_0}$.
\end{theorem}

\begin{theorem}[Bella,~\cite{Bel87}]\label{AB}
A ccc space with a $G_\delta$-diagonal of rank 2 has cardinality not exceeding $2^{\aleph_0}$.
\end{theorem}

The above two results make the following questions very interesting:

\begin{question}\label{question}
{\rm (a)} Is the cardinality of an almost Lindel\"of space with a regular $G_\delta$-diagonal at most $2^{\aleph_0}$?

{\rm (b)} Is the cardinality of an almost Lindel\"of space with a $G_\delta$-diagonal of rank 2 at most $2^{\aleph_0}$?
\end{question}

In relation to Question \ref{question}(b) we can prove the following theorem. For its proof we need to introduce some 
notation and to recall one well-known result.

The symbol $[X]^2$ indicates the collection of all two-point subsets of $X$. The countable version of the Erd\"os-Rado 
Theorem is the following statement:

\begin{lemma}\label{ER}
If $X$ is a set with $|X|>2^{\aleph_0}$ and $[X]^2=\Un\{F_n:n\in\omega\}$, then there exists an uncountable set 
$S\sse X$ and an integer $n_0$ such that $[S]^2\sse F_{n_0}$.
\end{lemma}

\begin{theorem}\label{TG}
If $X$ is an almost Lindel\"of space with a strong $G_\delta$-diagonal of rank 2, then $|X| \le 2^{\aleph_0}$.
\end{theorem}

\begin{proof}
Let $\{\scr{U}_n:n\in\omega\}$ be a sequence of open covers witnessing that $X$ has a strong $G_\delta$-diagonal of rank 
2 and suppose that $|X|>2^{\aleph_0}$. Put $F_n=\{\{x,y\}\in[X]^2: y\notin \cl{St(St(x,\scr{U}_n),\scr{U}_n)\}}$. (Note 
that $y\notin \cl{St(St(x,\scr{U}_n),\scr{U}_n)}$ if and only if $x\notin \cl{St(St(y,\scr{U}_n),\scr{U}_n)}$.)
Since $[X]^2=\Un\{F_n:n<\omega\}$, by Lemma~\ref{ER} there is an uncountable set $S\sse X$ and an integer $n_0$ 
such that $[S]^2\sse F_{n_0}$. Let $\scr{V}=\{St(x,\scr{U}_{n_0}):x\in S\}$. It is straightforward to verify that 
$\cl{\bigcup\scr{V}}\subset \bigcup\{St(St(x,\scr{U}_{n_0}),\scr{U}_{n_0}):x\in S\}$. Then 
$\{St(St(x,\scr{U}_{n_0}),\scr{U}_{n_0}):x\in S\}\cup (X\setminus\cl{\bigcup\scr{V}})$ is an open cover of $X$ which does not 
contain a countable subfamily $\scr{V}_0$ with the property that $X=\bigcup\{\cl{V}:V\in\scr{V}_0\}$. Since this contradicts the 
fact that $X$ is almost Lindel\"of, the proof is completed.
\end{proof}

In relation to the previous theorem we ask the following question:

\begin{question}
Is the cardinality of a weakly Lindel\"of space with a strong $G_\delta$-diagonal of rank 2 at most $2^{\aleph_0}$?
\end{question}

Recently Gotchev obtained the following result:

\begin{theorem}[\cite{Got}] 
A weakly Lindel\"of space with a regular $G_\delta$-diagonal has cardinality not exceeding $2^{\mathfrak{c}}$.
\end{theorem}

The latter result suggests the following:

\begin{question}
Is the cardinality of a weakly Lindel\"of space with a $G_\delta$-diagonal of rank 2 at most $2^{\mathfrak{c}}$?
\end{question}

Here we will show that for the class of normal spaces even a stronger inequality is true (see Corollary \ref{CBC}).

Notice that not every space with a $G_\delta$-diagonal of rank 2 has a regular $G_\delta$-diagonal. For example, every 
Moore space has a $G_\delta$-diagonal of rank 2 \cite{ABu}, while not every Moore space has a regular $G_\delta$-diagonal 
\cite{Gru84}. Also, by a theorem of McArthur~\cite{McA} every pseudocompact space with a regular $G_\delta$-diagonal is 
compact, while in~\cite{ABu} it is pointed out that the Mrowka space, which is a pseudocompact and not countably compact 
space, has a $G_\delta$-diagonal of rank 2. (Indeed, if 
$X=\omega\un\scr{A}$, where $\scr{A}$ is a MAD family in $\omega$, by letting 
$\scr{U}_n=\{\{k\}:k\in\omega\}\un\{\{A\}\un A\minus n:A\in\scr{A}\}$, we obtain a $G_\delta$-sequence of rank 2.)

Therefore, by McArthur's result, the above space also witnesses that $G_\delta$-diagonal of rank 2 does not imply regular 
$G_\delta$-diagonal. However, this argument does not work for normal spaces. A normal pseudocompact space is countably 
compact and Chaber's Theorem~\cite{EN} shows that within the class of countably compact spaces even 
$G_\delta$-diagonal of rank 1 suffices to imply compactness. This makes very interesting the following question:

\begin{question} Does there exist a normal space with a $G_\delta$-diagonal of rank 2 which does not have a regular 
$G_\delta$-diagonal?
\end{question}

The converse implication remains fully open:

\begin{question}[\cite{ABe}]
Is it true that every space with a regular $G_\delta$-diagonal also has a $G_\delta$-diagonal 
of rank 2?
\end{question}

We recall that a collection of non-empty open sets in $X$ is a \emph{discrete cellular family} if every $x\in X$ has an open 
neighborhood which intersects at most one of the members of the family.
A space $X$ has \emph{countable discrete cellularity} (briefly, cdc) if every discrete family of open sets in $X$ is 
countable. We give two proofs that every weakly Lindel\"of space is cdc.

\begin{lemma}
Every weakly Lindel\"of space is cdc.
\end{lemma}

\begin{proof} {\bf First proof:}
Let $X$ be a weakly Lindel\"of space and suppose that there exists a discrete family of open sets 
$\scr{U}=\{U_\alpha:\alpha<\omega_1\}$. Since $\scr{U}$ is discrete, for every $x\in X$ we can find an open 
neighborhood $V_x$ such that $V_x$ intersects at most one of the members of $\scr{U}$. Therefore 
$\cl{\bigcup\scr{U}}=\bigcup\{\cl{U_\alpha}:\alpha<\omega_1\}$, hence $\cl{\bigcup\scr{U}}$ is a regular closed subset 
of $X$. For every $\alpha<\omega_1$ and every $x\in \cl{U_\alpha}\setminus U_\alpha$ we fix a non-empty open set 
$V_x$ such that $V_x\cap U_\beta=\emptyset$ for every $\beta\ne\alpha$. Let $\scr{V}$ be the family of all such sets 
$V_x$. Then $\scr{U}\cup\scr{V}$ is an open cover of $\cl{\bigcup\scr{U}}$ which clearly does not have a countable 
subcolection which is dense in $\cl{\bigcup\scr{U}}$ -- contradiction.

{\bf Second proof:} Let $X$ be a weakly Lindel\"of space and suppose that there exists a discrete family of open sets 
$\scr{U}=\{U_\alpha:\alpha<\omega_1\}$. Since $\scr{U}$ is discrete, 
$\cl{\bigcup\scr{U}}=\bigcup\{\cl{U_\alpha}:\alpha<\omega_1\}$. Even more, if 
$\scr{U}_\beta=\{U_\alpha:\alpha<\omega_1, \alpha\ne\beta\}$, then 
$\cl{\bigcup\scr{U}_\beta}=\bigcup\{\cl{U_\alpha}:\alpha<\omega_1,\alpha\ne\beta\}$ and 
$\cl{\bigcup\scr{U}_\beta}\cap \cl{U_\beta}=\emptyset$.
Then for each $\beta<\omega_1$, the set $X\setminus \cl{\bigcup\scr{U}_\beta}$ is open and contains $\cl{U_\beta}$.
Thus $\{X\setminus \cl{\bigcup\scr{U}_\beta}:\beta<\omega_1\}$ is an open cover of $X$ which clearly does not have a 
countable subcolection which is dense in $X$ -- contradiction.
\end{proof}

Now let $X$ be a topological space and $\kappa=wL(X)$. By replacing $\omega_1$ in the above proof with $\kappa^+$ 
one can prove the following:

\begin{theorem}\label{TDC}
Let $X$ be a topological space and $\kappa=wL(X)$. Then every discrete cellular family in $X$ has cardinality at most 
$\kappa$.
\end{theorem}

In relation to the above observation the following cardinal function could be introduced.

\begin{definition}
The \emph{discrete cellularity} of $X$ is $dc(X):=\sup\{|\scr{U}|:\scr{U}$ is a discrete cellular family in $X\}+\omega$. If 
$dc(X)=\omega$, then it is called that $X$ has \emph{countable discrete celularity} ($X$ has \emph{cdc} or $X$ is a \emph{cdc} space). 
\end{definition}

Using the cardinal function $dc(X)$, Theorem \ref{TDC} could be restated as follows:

\begin{theorem}
For every topological space $X$, $dc(X)\le wL(X)$.
\end{theorem}

To formulate our next results in a more general setting, recall that a space $X$ is \emph{star--$\scr{P}$} if for every open 
cover $\scr{U}$ of $X$ there exists a subspace $Y$ of $X$ with the property $\scr{P}$ such that $X=St(Y,\scr{U})$.

First, by strengthening one of the hypothesis in Theorem \ref{TG} and weakening the other we can prove the following:

\begin{theorem}\label{TG2}
If $X$ is a star-cdc space with a $G_\delta$-diagonal of rank 3, then $|X| \le 2^{\aleph_0}$.
\end{theorem}

\begin{proof}
Let $\{\scr{U}_n:n\in\omega\}$ be a sequence of open covers witnessing that $X$ has a $G_\delta$-diagonal of rank 3 and 
suppose that $|X|>2^{\aleph_0}$. Put $F_n=\{\{x,y\}\in[X]^2: y\notin St(St(St(x,\scr{U}_n),\scr{U}_n),\scr{U}_n)\}$. 
(Note that $y\notin St(St(St(x,\scr{U}_n),\scr{U}_n),\scr{U}_n)$ if and only if 
$x\notin St(St(St(y,\scr{U}_n),\scr{U}_n),\scr{U}_n)$).
Since $[X]^2=\Un\{F_n:n<\omega\}$, by Lemma~\ref{ER} there is an uncountable set $S\sse X$ and an integer $n_0$ 
such that $[S]^2\sse F_{n_0}$. It follows from the fact that $\scr{U}_{n_0}$ is an open cover of $X$ and the definition of $S$ that the set $S$ is closed. 
Then $\scr{V}=\{St(x,\scr{U}_{n_0}):x\in S\}\cup(X\setminus S)$ is an open cover of $X$. Therefore, there is $Y\subset X$ 
such that $Y$ is cdc and $St(Y,\scr{V})=X$. It is straightforward to verify that 
$\{Y\cap St(x,\scr{U}_{n_0}):x\in S\}$ is an uncountable discrete cellular family in $Y$. Since this contradicts the 
fact that $Y$ is cdc, the proof is completed.
\end{proof}

\begin{corollary}
If $X$ is a cdc space with a $G_\delta$-diagonal of rank 3, then $|X| \le 2^{\aleph_0}$.
\end{corollary}

\begin{corollary}
If $X$ is a weakly Lindel\"of space with a $G_\delta$-diagonal of rank 3, then $|X| \le 2^{\aleph_0}$.
\end{corollary}

By strengthening one of the hypothesis in Theorem \ref{TG2} and weakening the other we can reach the same conclusion:

\begin{theorem}\label{nscdc}
A normal star-cdc space $X$ with a $G_\delta$-diagonal of rank 2 satisfies $|X|\leq 2^{\aleph_0}$. 
\end{theorem}

\begin{proof}
Let $\{\scr{U}_n:n\in\omega\}$ be a sequence of open covers witnessing that $X$ has a rank 2 diagonal and suppose 
that $|X|>2^{\aleph_0}$. Put $F_n=\{\{x,y\}\in[X]^2: St(x,\scr{U}_n)\meet St(y,\scr{U}_n)=\es\}$. 
Since $[X]^2=\Un\{F_n:n<\omega\}$, by Lemma~\ref{ER} there is an uncountable set $S\sse X$ and an integer $n_0$ 
such that $[S]^2\sse F_{n_0}$. It follows from the fact that $\scr{U}_{n_0}$ is an open cover of $X$ and the definition of $S$ that the set $S$ is closed. 
Therefore, we may pick an open set $V$ such that $S\sse V$ and $\overline{V}\sse\Un\{St(x,\scr{U}_{n_0}):x\in S\}$. Now let 
$\scr{U}=\{St(x,\scr{U}_{n_0})\meet V:x\in S\}\un\{X\minus S\}$. Clearly $\scr{U}$ is an open cover of $X$ and therefore 
there is $Y\subset X$ such that $Y$ has cdc and $St(Y,\scr{U})=X$. Then we must have 
$Y\meet St(x,\scr{U}_{n_0})\meet V\neq\es$ for each $x\in S$. But, 
$\{Y\meet St(x,\scr{U}_{n_0})\meet V:x\in S\}$ is a discrete family of open subsets of $Y$ and so $Y$ cannot have cdc. 
This contradiction finishes the proof.
\end{proof}

\begin{corollary}
A normal cdc space with a $G_\delta$-diagonal of rank 2 has cardinality not exceeding $2^{\aleph_0}$.
\end{corollary}

\begin{corollary}\label{CBC}
A normal weakly Lindel\"of space with a $G_\delta$-diagonal of rank 2 has cardinality not exceeding $2^{\aleph_0}$.
\end{corollary}

We note that in \cite[Theorem 2.2]{BS2020} the following more general result than the one in Corollary \ref{CBC} was proved:

\begin{theorem}
A normal dually weakly Lindel\"of space with a $G_\delta$-diagonal of rank 2 has cardinality not exceeding $2^{\aleph_0}$.
\end{theorem}

The following similar result was obtained in~\cite{BBR}, Proposition 3.4:

\begin{theorem}
If $X$ is a Baire space with a $G_\delta$-diagonal of rank 2, then $|X|\leq 2^{\aleph_0}$.
\end{theorem}

\section{A few improved cardinality bounds}\label{S4}

We begin with some significant improvements of the cardinality bounds given in the following theorem. 

\begin{theorem}[{\cite[Theorem 2.12]{BC2018}}]
Let $X$ be a space which is locally $H$-closed, locally Lindel\"of, locally ccc, or regular and locally normal. 
Then $|X|\le 2^{\chi(X)wL(X)}$.
\end{theorem}

As in~\cite{BC2018}, for a property $\scr{P}$ of a space X, we say X is \emph{locally} $\scr{P}$ if every point in X has a neighborhood with property $\scr{P}$.

In the proofs of some of the following theorems we will use the fact that if $X$ is a Hausdorff space, then 
$|X|\le d(X)^{\psi_c(X)t(X)}$ (see \cite{BC88}).

\begin{theorem}\label{th4.2}
If $X$ is locally $H$-closed, then $|X|\leq wL(X)^{\psi_c(X)t(X)}$.
\end{theorem}

\begin{proof}
For every $x\in X$ there exists an open set $U_x$ such that $x\in U_x$ and $\overline{U}_x$ is $H$-closed. It follows from 
the Dow-Porter bound~\cite{DowPor82} for $H$-closed spaces that for every $x\in X$, $|U_x|\leq |\overline{U}_x|\leq 2^{\psi_c(\overline{U}_x)}\leq 2^{\psi_c(X)}$.
Clearly $\scr{U}=\{U_x:x\in X\}$ is an open cover of $X$. Then there exists $\scr{V}\in[\scr{U}]^{\leq wL(X)}$ such that 
$X=\overline{\Un\scr{V}}$. Thus $\Un\scr{V}$ is dense in $X$ and $|\Un\scr{V}|\leq wL(X)\cdot 2^{\psi_c(X)}\leq wL(X)^{\psi_c(X)}$. As $X$ is Hausdorff, we have
\begin{align}
|X|\leq d(X)^{\psi_c(X)t(X)}\leq |\Un\scr{V}|^{\psi_c(X)t(X)}&\leq\left(wL(X)^{\psi_c(X)}\right)^{\psi_c(X)t(X)}\notag\\
&=wL(X)^{\psi_c(X)t(X)}.\notag
\end{align}
\end{proof}

\begin{corollary}
If $X$ is locally $H$-closed, then $|X|\leq wL(X)^{\chi(X)}$.
\end{corollary}

If $X$ is locally compact, we have the following corollary, which improves Corollary 2.13 in \cite{BC2018}. Recall that 
$\chi(X)=\psi(X)=\psi_c(X)$ for locally compact spaces and that $t(X)\le\chi(X)$ for every space $X$.

\begin{corollary}\label{lc}
If $X$ is locally compact, then $|X|\leq wL(X)^{\psi(X)}=wL(X)^{\chi(X)}$.
\end{corollary}

The following two theorems have proofs that follow a similar structure. 

\begin{theorem}\label{th4.4}
If $X$ is locally Lindel\"of, then $|X|\leq wL(X)^{\psi_c(X)t(X)}$.
\end{theorem}

\begin{proof}
For every $x\in X$ fix a neighborhood $F_x$ of $x$ which is Lindel\"of. Since $X$ is Hausdorff and $\psi(X)$ and $t(X)$ are monotone functions, for every $x\in X$ we have
$|\mathrm{Int}(F_x)|\le |F_x|\le 2^{L(F_x)\psi(F_x)t(F_x)}\le 2^{\psi(X)t(X)}\le 2^{\psi_c(X)t(X)}$.
Clearly $\scr{U}=\{\mathrm{Int}(F_x):x\in X\}$ is an open cover of $X$.
Hence, there exists $\scr{V}\in[\scr{U}]^{\leq wL(X)}$ such that $X=\overline{\Un\scr{V}}$.
Thus $\Un\scr{V}$ is dense in $X$ and $|\Un\scr{V}|\leq wL(X)\cdot 2^{\psi_c(X)t(X)}\leq wL(X)^{\psi_c(X)t(X)}$. As $X$ is Hausdorff, we have
\begin{align}
|X|\leq d(X)^{\psi_c(X)t(X)}\leq |\Un\scr{V}|^{\psi_c(X)t(X)}&\leq\left(wL(X)^{\psi_c(X)t(X)}\right)^{\psi_c(X)t(X)}\notag\\&=wL(X)^{\psi_c(X)t(X)}.\notag
\end{align}
\end{proof}

\begin{corollary}
If $X$ is locally Lindel\"of, then $|X|\leq wL(X)^{\chi(X)}$.
\end{corollary}

In the proof of the following theorem we use Sun's inequality that if $X$ is a Hausdorff space, then 
$|X|\le\pi\chi(X)^{c(X)\psi_c(X)}$ (see \cite{SH88}).

\begin{theorem}
If $X$ is locally ccc, then $$|X|\leq (wL(X)\pi\chi(X))^{\psi_c(X)t(X)}.$$
\end{theorem}

\begin{proof}
For every $x\in X$ fix a neighborhood $F_x$ of $x$ which is ccc.
Since $X$ is Hausdorff, $\psi_c(X)$ is monotone, and $c(X)$ and $\pi\chi(X)$ are monotone with respect to open sets, for every $x\in X$ we have
$|\mathrm{Int}(F_x)|\le \pi\chi(\mathrm{Int}(F_x))^{c(\mathrm{Int}(F_x))\psi_c(\mathrm{Int}(F_x))}\le \pi\chi(X)^{\psi_c(X)}$.
Clearly $\scr{U}=\{\mathrm{Int}(F_x):x\in X\}$ is an open cover of $X$.
Hence, there exists $\scr{V}\in[\scr{U}]^{\leq wL(X)}$ such that $X=\overline{\Un\scr{V}}$.
Thus $\Un\scr{V}$ is dense in $X$ and $|\Un\scr{V}|\leq wL(X)\cdot \pi\chi(X)^{\psi_c(X)}\leq (wL(X)\pi\chi(X))^{\psi_c(X)}$. As $X$ is Hausdorff, we have
\begin{align}
|X|\leq d(X)^{\psi_c(X)t(X)}\leq |\Un\scr{V}|^{\psi_c(X)t(X)}&\leq\left(wL(X)\pi\chi(X))^{\psi_c(X)}\right)^{\psi_c(X)t(X)}\notag\\&=(wL(X)\pi\chi(X))^{\psi_c(X)t(X)}.\notag
\end{align}
\end{proof}

\begin{corollary}
If $X$ is locally ccc, then $|X|\leq wL(X)^{\chi(X)}$.
\end{corollary}

In the power homogeneous (PH) setting, it was shown in Corollary 3.8 in \cite{BC2018} that if $X$ is locally compact and PH, then $|X|\leq 2^{wL(X)t(X)}$, an extension of De la Vega's Theorem for compact homogeneous spaces. The following 
theorem improves that corollary. 

\begin{theorem}\label{lcph}
If $X$ is a locally compact, power homogeneous space, then $|X|\leq wL(X)^{t(X)}$. 
\end{theorem}

\begin{proof}
For every $x\in X$ there exists an open set $U_x$ such that $x\in U_x$ and $\overline{U}_x$ is compact. Using Theorem 3.6 from \cite{BC2018}, 
which states that if $X$ is a power homogeneous space and $U \subseteq X$ is a non-empty open set, then $|U| \le 2^{L(\overline{U})t(X)pct(X)}$, 
we conclude that for each $x\in X$, $|U_x|\leq 2^{t(X)}$. (Note that $pct(X)$ is countable as $X$ is locally 
compact). Then $\scr{U}=\{U_x:x\in X\}$ is an open cover of $X$. Hence, there exists $\scr{V}\in[\scr{U}]^{\leq wL(X)}$ 
such that $X=\overline{\Un\scr{V}}$. Thus $\Un\scr{V}$ is dense in $X$ and 
$|\Un\scr{V}|\leq wL(X)\cdot 2^{t(X)}\leq wL(X)^{t(X)}$.

It was shown by Ridderbos in \cite{rid2006} that the cardinality of a power homogeneous Hausdorff space is at most 
$d(X)^{\pi\chi(X)}$. After noting that in our case $\pi\chi(X)\leq t(X)pct(X)=t(X)$, we conclude that
$$|X|\leq d(X)^{\pi\chi(X)}\leq |\Un\scr{V}|^{t(X)}\leq\left(wL(X)^{t(X)}\right)^{t(X)}=wL(X)^{t(X)}.$$
\end{proof}

We note that the proof of the above theorem is in fact simpler than the proof of Corollary 3.8 in \cite{BC2018}. The proof of 
that corollary there fundamentally relies on the main theorem in that paper (Theorem 2.3). However, the proof above does 
not.

Our next theorem shows that Bella--Carlson's inequality $|X|\le 2^{\psi(X)t(X)wL(X)}$, which was proved in \cite{BC2018} 
for every regular $T_1$-space $X$ with a $\pi$-base whose elements have compact closures, is also valid for Hausdorff 
spaces satisfying the same condition but after replacing in it $\psi(X)$ with $\psi_c(X)$. (We recall that a \emph{$\pi$-base} 
for a space $X$ is a collection $\scr V$ of non-empty open sets in $X$ such that if $U$ is any non-empty open set in $X$, 
then there exists $V \in \scr V$ such that $V \subset U$.) 

First we make the following observation.

\begin{lemma}\label{LCPB}
If $X$ is a space with a $\pi$-base $\scr{B}$ whose elements have compact closures and $F\subsetneq X$ is 
closed, then there exists $B\in\scr{B}$ such that $\cl{B}\cap F=\emptyset$.
\end{lemma}

\begin{proof}
Since $F$ is closed and $X\setminus F\ne\emptyset$, there is $B_1 \in \scr{B}$ such that $B_1 \subseteq X\setminus F$. 
If $\cl{B_1}\cap F=\emptyset$, then $B=B_1$. If $\cl{B_1}\cap F\ne\emptyset$, then clearly 
$\cl{B_1}\cap F \subset \cl{B_1}\setminus B_1$. Let $L=\cl{B_1}\setminus B_1$ and $y\in B_1$. 
Then $L$ is closed in $\cl{B_1}$ and since 
$\cl{B_1}$ is compact and Hausdorff there is an open set $B_2$ in $\cl{B_1}$ such that $y\in B_2$ and 
$\cl{B_2}\cap L=\emptyset$. Then $B_2\subset B_1$ and since $B_1$ is open in $X$, $B_2$ is also open in $X$. 
Since $\scr{B}$ is a $\pi$-base in $X$, there is a non-empty open set $B_3\in\scr{B}$ such that $B_3\subset B_2$. Clearly, 
$\cl{B_3}\cap F=\emptyset$ and we can choose $B=B_3$.
\end{proof}

\begin{theorem}\label{TBCG}
If $X$ is a space with a $\pi$-base whose elements have compact closures, then 
$|X|\le 2^{\psi_c(X)t(X)wL(X)}$.
\end{theorem}

\begin{proof}
Let $\kappa=\psi_{c}(X)t(X)wL(X)$ and let $\scr{B}$ be a $\pi$-base of non-empty open sets with compact closures. 
Notice that for each $B\in\scr{B}$ if $\psi(\cl{B})$ is the pseudocharacter of the space $\cl{B}$, then we have 
$\psi(\cl{B})\le\psi_c(X)$ and since $\cl{B}$ is compact and Hausdorff, we have 
$\cl{B}\le 2^{\psi(\cl{B})}\le 2^\kappa$. Since $\psi_c(X)\le \kappa$, for each $x\in X$ we can fix a collection $\scr{V}_x$ 
of open neighborhoods of $x$ such that $|\scr{V}_x|\le \kappa$ and $\bigcap\{\cl{V}:V\in \scr{V}_x\}=\{x\}$. Without 
loss of generality we may assume that each $\scr{V}_x$ is closed under finite intersections.

We will construct by transfinite reqursion a non-decreasing chain of open sets $\{U_\alpha : \alpha < \kappa^+\}$ such that 
\begin{itemize}
\item[(1)] $\cl{U_\alpha} \le 2^\kappa$ for every $\alpha < \kappa^+$, and 
\item[(2)] if $X\setminus\cl{\bigcup \scr{M}} \ne \emptyset$ for some 
$\scr{M} \in [\bigcup\{\scr{V}_x : x \in \cl{U_\alpha}\}]^{\le \kappa}$, 
then there is $B_\scr{M}\in\scr{B}$ such that $B_\scr{M}\subset U_{\alpha+1}\setminus\cl{\bigcup \scr{M}}$.
\end{itemize}

Let $B_0\in\scr{B}$ be arbitrary. We set $U_0=B_0$.  Then $|\cl{U_0}|\le 2^\kappa$. 
If $\beta =\alpha + 1$, for some $\alpha$, then for every 
$\scr{M} \in [\bigcup\{\scr{V}_x : x \in \cl{U_\alpha}\}]^{\le \kappa}$ such that $X\setminus\cl{\bigcup\scr{M}} \ne \emptyset$, 
we choose $B_{\scr{M}} \in \scr{B}$ such that $B_{\scr{M}} \subseteq X\setminus\cl{\bigcup \scr{M}}$. 
We define $U_{\beta} = U_\alpha\cup\bigcup\{B_\scr{M} : \scr{M} \in [\bigcup\{\scr{V}_x : x \in \cl{U_\alpha}\}]^{\le \kappa}, X\setminus\cl{\bigcup\scr{M}}\ne\emptyset\}$. Therefore, since $X$ is Hausdorff, we have  
$|\cl{U_{\beta}}| \le 2^\kappa$. If $\beta<\kappa^+$ is a limit ordinal we let 
$U_\beta = \bigcup_{\alpha<\beta} U_\alpha$. Then clearly $|U_\beta|\le 2^\kappa$, hence 
$|\cl{U_\beta}| \le 2^\kappa$.

Let $F = \bigcup\{\cl{U_\alpha} : \alpha < \kappa^+\}$. Then $|F| \le 2^\kappa$. 
Since $t(X)\le\kappa$, $F$ is closed and therefore $F = \cl{\bigcup\{U_\alpha : \alpha < \kappa^+\}}$. 
Thus, $F$ is a regular-closed set. 

We will show that $X = F$. Suppose that $X \ne F$. Since $F$ is closed and $\scr{B}$ is a $\pi$-base, there is $B\in\scr{B}$
such that $B\subset X\setminus F$ and $\cl{B}\cap F=\emptyset$ (Lemma \ref{LCPB}).
Then for every $x\in F$ and $y\in\cl{B}$ there is 
$V_x(y)\in \scr{V}_x$ such that $y\notin \cl{V_x(y)}$. Therefore, using the compactness of $\cl{B}$ and the fact that each 
$\scr{V}_x$ is closed under finite intersections, for every $x\in F$ we can find 
$V_x \in \scr{V}_x$ such that $\cl{V_x} \cap \cl{B}=\emptyset$; hence $\cl{B}\cap V_x=\emptyset$.
Clearly $\{V_x : x \in F\}$ is an open cover of $F$. Since $wL(X)$ is monotone with respect to regular-closed sets, there
exists $\scr{M} \in \{V_x : x \in F\}^{\le \kappa}$ such that $F \subseteq \cl{\bigcup\scr{M}}$. Then there 
exists $\alpha < \kappa^+$ such that $\scr{M} \in [\bigcup\{\scr{V}_x : x \in \cl{U_\alpha}\}]^{\le \kappa}$. 
As $\cl{B} \cap \bigcup \scr{M} = \emptyset$ it follows that $B\subset X\setminus\cl{\bigcup\scr{M}}$, hence 
$X\setminus\cl{\bigcup\scr{M}}\ne\emptyset$. Thus, there exists $B_\scr{M}\in\scr{B}$ such that 
$\emptyset\ne B_\scr{M} \subseteq U_{\alpha+1}\setminus\cl{\bigcup\scr{M}} \subseteq F\setminus\cl{\bigcup\scr{M}} = \emptyset$. 
Since this is a contradiction, we conclude that $X = F$ and the proof is completed.
\end{proof}

\begin{corollary} If $X$ is a Hausdorff space with a dense set of isolated points, then 
$|X|\le 2^{\psi_c(X)t(X)wL(X)}$.
\end{corollary}

\begin{corollary} Let $X$ be a space with a dense set of isolated points.

{\rm (a) [Dow--Porter, 1982 \cite{DowPor82}]} If $X$ is Hausdorff, then $|X|\le 2^{\chi(X)wL(X)}$.

{\rm (b) [Alas, 1993 \cite{Ala93}]} If $X$ is Hausdorff, then $|X|\le 2^{\psi_c(X)t(X)wL_c(X)}$.

{\rm (c)  [Bella--Carlson, 2017 \cite{BC2018}]} If $X$ is regular and $T_1$, then $|X|\le 2^{\psi(X)t(X)wL(X)}$.
\end{corollary}

\begin{lemma}\label{LBC}
Let $X$ be a space with a dense subset $D \subseteq X$ such that each $d \in D$ has a
closed neighborhood that is compact, $H$-closed, normal, Lindel\"of, or ccc. Then $X$ has a $\pi$-base $\scr{B}$ such that
$\cl{B}$ is respectively compact, $H$-closed, normal, Lindel\"of, or ccc, whenever $B\in\scr{B}$.
\end{lemma}

\begin{proof}
Let $U$ be a non-empty open set in $X$. Then there exists
$d \in U \cap D$ and an open set $V$ containing $d$ such that $\cl{V}$ is compact, $H$-closed,
normal, Lindel\"of, or ccc. Let $B = U\cap V$. As the $H$-closed
and c.c.c properties are hereditary on regular-closed sets, and compactness, normality
and Lindel\"ofness are hereditary on closed sets, we conclude that $\cl{B}$ is respectively compact, $H$-closed,
normal, Lindel\"of, or ccc. As $B \subseteq U$, this shows that $X$ has a $\pi$-base $\scr{B}$ such that
$\cl{B}$ is respectively compact, $H$-closed, normal, Lindel\"of, or ccc, whenever $B\in\scr{B}$.
\end{proof}

As a direct corollary of Theorem \ref{TBCG} and Lemma \ref{LBC} we obtain the following theorem:

\begin{theorem}
If $X$ is a space with a dense subset $D \subseteq X$ such that each $d \in D$ has a
compact neighborhood, then $|X|\le 2^{\psi_c(X)t(X)wL(X)}$.
\end{theorem}

In relation to spaces with $\pi$-bases whose elements have compact closures we can make also the following observation for which we need to 
recall the following definitions: The \emph{$\theta$-closure} of a set $A$ in a space $X$, denoted by $\mathrm{cl}_\theta(A)$, is the set of all points 
$x\in X$ such that for every open neighborhood $U$ of $x$ we have $\overline{U}\cap A \ne \emptyset$; $A$ is
called \emph{$\theta$-dense} in $X$ if $\mathrm{cl}_\theta(A)=X$; and
the $\theta$-density of a space $X$ is $d_\theta(X)=\min\{|A|:A\subset X, \mathrm{cl}_\theta(A)=X\}$.

\begin{theorem}\label{TCTC}
If $X$ is a space with a $\pi$-base whose elements have compact closures, then $d(X)=d_\theta(X)$.
\end{theorem}

\begin{proof}
If $D\subseteq X$ is dense in $X$, then $D$ is $\theta$-dense. Therefore $d_\theta(X)\le d(X)$.

Now let $D$ be $\theta$-dense in $X$ and suppose that $D$ is not dense in $X$. Then $\cl{D}\ne X$. Since $\cl{D}$ is a 
closed subset of $X$, if follows from Lemma \ref{LCPB} that there exists element $B$ of the $\pi$-base such that 
$\cl{B}\cap \cl{D}=\emptyset$. If $x\in B$, then $x\notin \mathrm{cl}_\theta(D)$ -- contradiction. Thus $D$ is dense in $X$ 
and therefore $d(X)\le d_\theta(X)$.
\end{proof}

Since $H$-closedness is a natural generalization of compactness, the following question is natural.

\begin{question}
Let $X$ be a Hausdorff space with a $\pi$-base $\scr{B}$ such that $\cl{B}$ is $H$-closed for every $B\in\scr{B}$. Is it 
true that $|X|\le 2^{\psi_c(X)t(X)wL(X)}$?
\end{question}

In relation to the above question we can prove the following:

\begin{theorem}\label{TBCG2}
If $X$ is a Hausdorff space with a $\pi$-base whose elements have $H$-closed closures, then 
$d_\theta(X)\le 2^{\psi_c(X)t(X)wL(X)}$.
\end{theorem}

\begin{proof}
Let $\kappa=\psi_{c}(X)t(X)wL(X)$ and let $\scr{B}$ be a $\pi$-base of non-empty open sets with $H$-closed closures. 
Notice that for each $B\in\scr{B}$ if $\psi_c(\cl{B})$ is the closed pseudocharacter of the space $\cl{B}$ and since $\cl{B}$ 
is $H$-closed, we have $\cl{B}\le 2^{\psi_c(\cl{B})}\le 2^\kappa$. Since $\psi_c(X)\le \kappa$, for each $x\in X$ we can fix a 
collection $\scr{V}_x$ of open neighborhoods of $x$ such that $|\scr{V}_x|\le \kappa$ and 
$\bigcap\{\cl{V}:V\in \scr{V}_x\}=\{x\}$. Without loss of generality we may assume that each $\scr{V}_x$ is closed under 
finite intersections.

We will construct by transfinite recursion a non-decreasing chain of open sets $\{U_\alpha : \alpha < \kappa^+\}$ such that 
\begin{itemize}
\item[(1)] $\cl{U_\alpha} \le 2^\kappa$ for every $\alpha < \kappa^+$, and 
\item[(2)] if $X\setminus\cl{\bigcup \scr{M}} \ne \emptyset$ for some 
$\scr{M} \in [\bigcup\{\scr{V}_x : x \in \cl{U_\alpha}\}]^{\le \kappa}$, 
then there is $B_\scr{M}\in\scr{B}$ such that $B_\scr{M}\subset U_{\alpha+1}\setminus\cl{\bigcup \scr{M}}$.
\end{itemize}

Let $B_0\in\scr{B}$ be arbitrary. We set $U_0=B_0$.  Then $|\cl{U_0}|\le 2^\kappa$. 
If $\beta =\alpha + 1$, for some $\alpha$, then for every 
$\scr{M} \in [\bigcup\{\scr{V}_x : x \in \cl{U_\alpha}\}]^{\le \kappa}$ such that $X\setminus\cl{\bigcup\scr{M}} \ne \emptyset$, 
we choose $B_{\scr{M}} \in \scr{B}$ such that $B_{\scr{M}} \subseteq X\setminus\cl{\bigcup \scr{M}}$. 
We define $U_{\beta} = U_\alpha\cup\bigcup\{B_\scr{M} : \scr{M} \in [\bigcup\{\scr{V}_x : x \in \cl{U_\alpha}\}]^{\le \kappa}, X\setminus\cl{\bigcup\scr{M}}\ne\emptyset\}$. Therefore, since $X$ is Hausdorff, we have  
$|\cl{U_{\beta}}| \le 2^\kappa$. If $\beta<\kappa^+$ is a limit ordinal we let 
$U_\beta = \bigcup_{\alpha<\beta} U_\alpha$. Then clearly $|U_\beta|\le 2^\kappa$, hence 
$|\cl{U_\beta}| \le 2^\kappa$.

Let $F = \bigcup\{\cl{U_\alpha} : \alpha < \kappa^+\}$. Then $|F| \le 2^\kappa$. 
Since $t(X)\le\kappa$, $F$ is closed and therefore $F = \cl{\bigcup\{U_\alpha : \alpha < \kappa^+\}}$. 
Thus, $F$ is a regular-closed set. 

We will show that $X = \mathrm{cl}_\theta(F)$. Suppose that $X \ne \mathrm{cl}_\theta(F)$ and let 
$y_0\in X \setminus \mathrm{cl}_\theta(F)$. Then there is a neighborhood $U$ of $y_0$ such that 
$\cl{U}\cap F=\emptyset$ and since $\scr{B}$ is a $\pi$-base, there is $B\in\scr{B}$
such that $B\subset U$, hence $\cl{B}\cap F=\emptyset$.
Then for every $x\in F$ and $y\in\cl{B}$ there is 
$V_x(y)\in \scr{V}_x$ such that $y\notin \cl{V_x(y)}$. Therefore, using the $H$-closedness of $\cl{B}$ and the fact that each 
$\scr{V}_x$ is closed under finite intersections, for every $x\in F$ we can find 
$V_x \in \scr{V}_x$ such that $\cl{V_x} \cap \cl{B}=\emptyset$; hence $\cl{B}\cap V_x=\emptyset$.
Clearly $\{V_x : x \in F\}$ is an open cover of $F$. Since $wL(X)$ is monotone with respect to regular-closed sets, there
exists $\scr{M} \in \{V_x : x \in F\}^{\le \kappa}$ such that $F \subseteq \cl{\bigcup\scr{M}}$. Then there 
exists $\alpha < \kappa^+$ such that $\scr{M} \in [\bigcup\{\scr{V}_x : x \in \cl{U_\alpha}\}]^{\le \kappa}$. 
As $\cl{B} \cap \bigcup \scr{M} = \emptyset$ it follows that $B\subset X\setminus\cl{\bigcup\scr{M}}$, hence 
$X\setminus\cl{\bigcup\scr{M}}\ne\emptyset$. Thus, there exists $B_\scr{M}\in\scr{B}$ such that 
$\emptyset\ne B_\scr{M} \subseteq U_{\alpha+1}\setminus\cl{\bigcup\scr{M}} \subseteq F\setminus\cl{\bigcup\scr{M}} = \emptyset$. 
Since this is a contradiction, we conclude that $X = \mathrm{cl}_\theta(F)$ and the proof is completed.
\end{proof}

We note that Theorem \ref{TBCG} follows immediately from Theorem \ref{TBCG2} and Theorem \ref{TCTC}. Also, as a 
direct collorary of Theorem \ref{TBCG2} and Lemma \ref{LBC} we obtain the following:

\begin{theorem}
If $X$ is a space with a dense subset $D \subseteq X$ such that each $d \in D$ has a
$H$-closed neighborhood, then $d_\theta(X)\le 2^{\psi_c(X)t(X)wL(X)}$.
\end{theorem}

\end{document}